\documentclass[12pt]{amsart}
\usepackage{amssymb}
\usepackage[all]{xy}
\usepackage{fullpage}

\theoremstyle{plain}
\newtheorem{theorem}{Theorem}
\newtheorem{lemma}[theorem]{Lemma}
\newtheorem{proposition}[theorem]{Proposition}

\theoremstyle{definition}
\newtheorem{definition}[theorem]{Definition}
\newtheorem{remark}[theorem]{Remark}

\newcommand{\ZZ}{{\mathbb Z}}
\newcommand{\GG}{{\mathcal G}}
\newcommand{\rise}{{\operatorname{r}}}
\newcommand{\risee}{{\operatorname{r_E}}}
\newcommand{\risev}{{\operatorname{r_V}}}
\renewcommand{\AA}{{\mathcal A}}
\newcommand{\acts}{\curvearrowright}

\begin{document}

\title{Orders on trees and free products of left-ordered groups}

\author{Warren Dicks}
\address{Departament de Matem{\`a}tiques, Universitat Aut{\`o}noma de Barcelona, E-08193 Bellaterra (Barcelona), SPAIN}
\email{dicks@mat.uab.cat}
\thanks{The first-named author's research was partially supported by Spain's Ministerio de Ciencia e Innovaci\'on through
 Project MTM2011-25955}

\author{Zoran {\v{S}}uni{\'c}}
\address{Department of Mathematics, Hofstra University, 306 Roosevelt Hall, Hempstead, NY 11549, USA}
\email{zoran.sunic@hofstra.edu}
\thanks{The second-named author's research was partially supported by the National Science Foundation under Grant
No.~DMS-1105520}


\begin{abstract}
We construct total orders on the vertex set of an oriented tree. 
The orders are based only on up-down counts at the interior vertices and the edges along the unique geodesic from a given vertex to another.

As an application, we provide a short proof (modulo Bass-Serre theory) of Vinogradov's result that the free product of left-orderable groups is left-orderable.
\end{abstract}

\maketitle

\section{Introduction}\label{sec:intro}

In 1949, A. A. Vinogradov~\cite{vinogradov:ordered} used groups of positive units of ordered rings to prove
\begin{center}
(*)  free products of orderable groups are orderable.
\end{center}
In 1977,  D. S. Passman~\cite[Theorem 13.2.7]{passman:group-rings} made it explicit that Vinogradov's argument also shows
\begin{center}
($\dag$) free products of left-orderable groups are left-orderable.
\end{center}

In 1967, R. E. Johnson~\cite{johnson:ordered} simplified Vinogradov's proof of~(*) by using different ordered rings. 
In 1972, C.~Holland and E.~Scrimger~\cite[Theorem 3.1]{holland-scrimger:ordered} made it explicit that Johnson's argument shows~($\dag$), while~R.~G.~Burns and V.~W.~D.~Hale ~\cite[first paragraph]{burns-hale:group-rings} proved~($\dag$) by using the freeness of the kernel of the natural map from a free product to a direct product.
In 1990, G. M. Bergman~\cite[Theorem 16]{bergman:ordered} further simplified Vinogradov's proof of~(*) by using different ordered rings. 

The main purpose of this article is to present yet another proof of~($\dag$). 
More generally, given a  group $G$ acting on an oriented tree $T$ with trivial edge stabilizers and left-ordered vertex stabilizers,  we construct a $G$-invariant order on the vertex set of $T$. 
Applying this construction to the barycentric subdivision of $T$ yields a $G$-invariant order on $T$, that is, on the disjoint union of the vertex set and the edge set. 
Hence,  by Bass-Serre theory,  for any graph of groups in which each edge group is trivial and each vertex group is left-orderable, the fundamental group of this graph of groups is left-orderable; this is another formulation of~($\dag$). 
This proof has the advantages that it is quite simple,  granted the existence of  Bass-Serre trees, and also gives rise to explicit descriptions of positive cones for the fundamental groups in question.

In Section~\ref{sec:trees}, we deal with the case where $G$ is trivial. 
Explicitly, given any oriented tree $T$ together with a total order on the link of each vertex (that is, the set of all edges adjacent to said vertex),  we define a total order on the vertex set of $T$ using only up-down counting at the interior vertices and the edges along the geodesic from a given vertex to another.

In Section~\ref{s:order-products},  we  combine our order with group actions, using naturality. 
For concreteness, we prove~($\dag$) by considering a Bass-Serre tree.  
This gives an explicit order for free products that is easy to state, and we then give the description of the positive cone in Subsection~\ref{subsec:explicitfreeproduct}.

In Subsection~\ref{subsec:explicitfreegroup}, we consider the  Bass-Serre trees that are Cayley graphs of finitely generated free groups, and
find that the resulting total orders on the free groups coincide with those defined in~\cite{sunic:free-lex} and~\cite{sunic:from-oriented}.

In Subsection~\ref{subsec:rooted}, we discuss a connection to depth-first searches on rooted trees.

\begin{remark}
Our approach gives constructive information  about groups acting on trees with trivial edge stabilizers. 
For the much more complicated case of arbitrary edge stabilizers, I.~M. Chiswell~\cite{chiswell:ordered} applied important work of V.~V. Bludov and A.~M.~W. Glass~\cite{bludov-glass:ordered} to give a non-constructive proof that a group $G$ acting on a tree $T=(V,E)$ is left-orderable if (and only if) there exists a family $(\mathcal{R}_v:v \in V)$ such that, for each $v \in V$, the following hold:
(1) $\mathcal{R}_v$ is a nonempty set of left-invariant orders on the $G$-stabilizer $G_v$;
(2) for each $g \in G$, $\null^g(\mathcal{R}_v) = \mathcal{R}_{gv}$;
(3) for each $r \in \mathcal{R}_v$ and each edge $e$ adjacent to $v$, if $w$ is the other vertex of $e$, then the
restriction of $r$ to $G_e$ extends to some element of $\mathcal{R}_w$.
\end{remark}

\section{Ordering trees}\label{sec:trees}

In this section, we describe a total order on the vertex set of a tree in terms of an orientation together with a total order on the set of edges adjacent to each vertex, by  up-down counting along the geodesic from a given vertex to another.

An oriented tree $T=(V,E)$ is a tree (a nonempty connected graph with no cycles) in which every edge has an assigned orientation. 
Thus, one of the endpoints of every edge $e$ is declared the origin, denoted by $\operatorname{o}(e)$, and the other the terminus, denoted by $\operatorname{t}(e)$.
To each edge $e$ we associate an edge $e^{-1}$ (this is not an edge in $E$), for which $\operatorname{o}(e^{-1})=\operatorname{t}(e)$ and
$\operatorname{t}(e^{-1})=\operatorname{o}(e)$, and we call
it the edge inverse to $e$. For $e \in E$, we set $(e^{-1})^{-1} = e$ and declare $e$ inverse to $e^{-1}$. The edges in $E$ are called positively oriented, and their inverse edges negatively oriented.

A geodesic of length $n \geq 0$ in the tree $T$ is a sequence
\begin{equation}\label{e:path}
 p = v_0~e_1^{\varepsilon_1}~v_1~e_2^{\varepsilon_2}~v_2~\dots~v_{n-1}~e_n^{\varepsilon_n}~v_n
\end{equation}
such that $v_0,\dots,v_n$ are distinct vertices, $e_1,\dots,e_n$ are distinct edges in $E$, $\varepsilon_i=\pm 1$, for
$i=1,\dots,n$, and 
\begin{align*}
 \operatorname{o}(e_i^{\varepsilon_i})     &= v_{i-1} \\
 \operatorname{t}(e_{i}^{\varepsilon_{i}}) &= v_{i},  
\end{align*}
for $i=1,\dots,n$. We say that the geodesic $p$ given in~\eqref{e:path} is a geodesic from the vertex $v_0$ to the vertex $v_n$. For any  vertices $x$ and $y$ of $T$, there exists a unique geodesic in $T$, denoted by $p_{xy}$, from $x$ to $y$.

We say that a rise occurs at the edge $e_i$, $i=1,\dots,n$, along the geodesic $p$ given in~\eqref{e:path} if $\varepsilon_i =1$ and that a fall occurs at that edge if $\varepsilon_i=-1$. 
We define the \emph{edge-rise index} $\risee(p)$ along the geodesic $p$ given in~\eqref{e:path} by
\[
 \risee(p)= \#(\textup{rises at edges along }p) - \#(\textup{falls at edges along }p) = \sum_{i=1}^n{\varepsilon_i}.
\]

Assume that, for each vertex $v \in V$, a total order $\preceq_v$, called the \emph{local order} at $v$, is given on the set $E_v = \{e\in E \mid \operatorname{o}(e)=v \text{ or } \operatorname{t}(e)=v\}$ of edges adjacent to $v$. 
We say that a rise occurs at the vertex $v_i$, $i=1,\dots,n-1$, along the geodesic $p$ given in~\eqref{e:path} if
$e_i \prec_{v_i} e_{i+1}$ and that a fall occurs at that vertex if $e_i \succ_{v_i} e_{i+1}$. 
We define the \emph{vertex-rise index} $\risev(p)$ along the geodesic $p$ as the difference
\[
 \risev(p)= \#(\textup{rises at vertices along }p) - \#(\textup{falls at vertices along }p)
\]
between the number of rises and the number of falls encountered at the vertices along the geodesic $p$ (note that the extremal vertices $v_0$ and $v_n$ play no role in the vertex-rise index, only the interior vertices along the geodesic matter).

\begin{definition}\label{d:m}
Let $T=(V,E)$ be an oriented tree with a local  order at every vertex. 
For vertices $x,y \in V$, define the \emph{rise index} $\rise(x,y)$ as the sum of the edge-rise index and the vertex-rise index along the geodesic $p_{xy}$ from $x$ to $y$, i.e.,
\[
 \rise(x,y) =  \risee(p_{xy}) + \risev(p_{xy}).
\]
\end{definition}

\begin{theorem}\label{t:tree-order}
Let $T=(V,E)$ be an oriented tree with a local order at every vertex. 
The binary relation $\leq$ defined on the set $V$ by
\[
 x \leq y \quad \iff \quad \rise(x,y) \geq 0
\]
is a total order on $V$.
\end{theorem}

The proof is based on the following result.

\begin{lemma}\label{l:order}
Let $V$ be a set and $m: V \times V \to \ZZ$ a function such that, for all $x, y \in V$,
\begin{gather}
x \ne y \implies m(x,y)\ne 0, \label{e:1} \tag{i}\\
 m(x,y) = -m(y,x), \label{e:2} \tag{ii}
\end{gather}
and for all $x,y,z \in V$,
\begin{equation}\label{e:4} \tag{iii}
 m(x,y) + m(y,z) + m(z,x) \le 1.
\end{equation}
Then the relation $\leq$ defined on $V$ by
\[
 x \leq y \quad \iff \quad m(x,y) \geq 0
\]
is a total order on $V$.
\end{lemma}

\begin{proof}
Reflexivity follows from~\eqref{e:2}, while anti-symmetry follows from~\eqref{e:1} and~\eqref{e:2}.
If $x,y,z\in V$ are such that $x < y$ and $y < z$, then $m(x,y) \stackrel{\textup{(i)}}{\geq} 1$, $m(y,z)
\stackrel{\textup{(i)}}{\geq} 1$, and therefore
\[
 m(x,z) \stackrel{\textup{(iii)}}{\geq} m(x,z) + (m(x,y) +m(y,z)+m(z,x) - 1) 
\stackrel{\textup{(ii)}}{=} m(x,y)+m(y,z)-1 \geq 1,
\]
which means that $x < z$. 
Thus $\leq$ is an order on $V$. 
The order is total by~\eqref{e:2}, which ensures that at least one of the integers $m(x,y)$ and $m(y,x)$ must be nonnegative.
\end{proof}

\begin{proof}[Proof of Theorem~\ref{t:tree-order}]
We just need to verify that the rise index function $\rise: V \times V \to \ZZ$ from Definition~\ref{d:m} satisfies the conditions of Lemma~\ref{l:order}.

For $x, y \in V$, suppose that $x \ne y$, and let $p=p_{xy}$, as given in~\eqref{e:path}, be the geodesic from $x$ to $y$. 
Then $n \ge 1$. 
Since each of the $n$ edges and each of the $n-1$ interior vertices contribute $\pm 1$ to the rise index, and $n+(n-1)$ is odd, the rise index $\rise(x,y)$ is odd, and, hence, nonzero.

For $x, y \in V$, the geodesic  $p_{xy}$  from $x$ to $y$ has the same edges and vertices as the geodesic $p_{yx}$ from $y$ to $x$,  but in reversed order and with opposite edge orientations. 
Therefore the sign of the contribution of each edge to the rise index is switched, as is the sign of each vertex contribution. 
Therefore, $\rise(x,y) = - \rise(y,x)$.

For  $x, y, z \in V$,  it remains to show that the sum 
\[ 
 s=\rise(x,y)+\rise(y,z)+\rise(z,x)
\]
is at most $1$. 
If two of $x, y, z$ are equal, then it follows from the previous paragraph that $s=0$. 
Thus, we may assume that  $x$, $y$ and $z$ are three distinct elements of $V$, and it suffices to show that $s = \pm 1$.

The three geodesics $p_{xy}$, $p_{yz}$, and $p_{zx}$ either form a tripod, as in Figure~\ref{f:tripod}, or two of these three geodesics are subgeodesics of the inverse of the third, say $p_{xy}$ and $p_{yz}$ are subgeodesics of $p_{zx}^{-1}=p_{xz}$, with overlap $y$, as in Figure~\ref{f:3paths-no}.

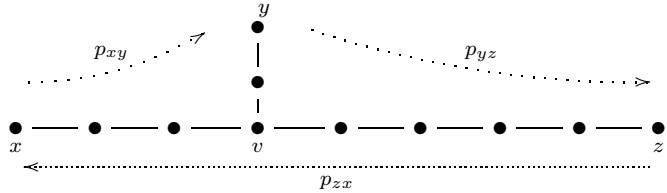
\begin{figure}[!ht]
\[
\xymatrix@R=5pt@C=5pt{
 &&&&&& \stackrel{\mkern10mu y}{\bullet} \ar@{-}[d] &
 \ar@{..>}@/_/[drrrrrrrrr]^{p_{yz}}   &&&&&&&&
 \\
 \ar@{..>}@/_/[urrrrr]^{p_{xy}} &&&&&& \bullet \ar@{-}[d] &&&&&&&&&&
 \\
 \bullet \ar@{-}[rr] \ar@{}[d]|{x} && \bullet \ar@{-}[rr] &&
 \bullet \ar@{-}[rr] && \bullet \ar@{-}[rr] \ar@{}[d]|{v}&&
 \bullet \ar@{-}[rr] && \bullet \ar@{-}[rr] &&
 \bullet \ar@{-}[rr] && \bullet \ar@{-}[rr] && \bullet \ar@{}[d]|{z}
 \\
 &&&&&&&& &&&&&&&& \ar@{..>}[llllllllllllllll]^{p_{zx}}
}
\]
\caption{Tripod case: $p_{xy}$, $p_{yz}$, and $p_{zx}$ form a tripod}
\label{f:tripod}
\end{figure}

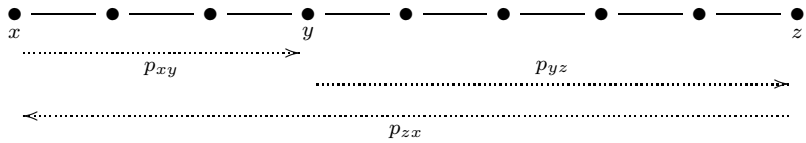
\begin{figure}[!ht]
\[
\xymatrix@R=5pt{
 \bullet \ar@{-}[r] \ar@{}[d]|{x} & \bullet \ar@{-}[r] & \bullet \ar@{-}[r]&
 \bullet \ar@{-}[r] \ar@{}[d]|{y} & \bullet \ar@{-}[r] & \bullet \ar@{-}[r]&
 \bullet \ar@{-}[r]               & \bullet \ar@{-}[r] & \bullet \ar@{}[d]|{z}
 \\
 \ar@{..>}[rrr]_{p_{xy}} &&&& &&&& \\
 &&& \ar@{..>}[rrrrr]^{p_{yz}} &&&&& \\
 &&&& &&&& \ar@{..>}[llllllll]^{p_{zx}}
}
\]
\caption{Line-segment case with overlap $y$: $p_{xy}$ and $p_{yz}$ are subgeodesics of $p_{zx}^{-1}$}
\label{f:3paths-no}
\end{figure}

In the tripod case (Figure~\ref{f:tripod}) the edge-rise contributions to the sum $s$ cancel (each edge in the tripod is traversed once in each direction), as do all vertex-rise contributions except for the three at the vertex $v$. 
However, two of the contributions at $v$ cancel. 
Indeed, there are three edges in the tripod that are adjacent to $v$ and one is the largest among them in the local order at $v$. The two vertex-rise contributions at $v$ obtained by entering and exiting this largest edge cancel. 
Therefore $s = \pm1$.

In the line-segment case (Figure~\ref{f:3paths-no}) the edge-rise contributions to the sum $s$ cancel, as do all vertex-rise contributions except for the one at the vertex $y$ along the geodesic $p_{zx}$. Therefore $s = \pm1$.
\end{proof}

\begin{remark}
We do not require that trees are locally finite (this is important for our application to free products).
\end{remark}

\begin{remark}\label{r:path}
All edges in the geodesic $p$ given in~\eqref{e:path} are positively oriented if and only if 
\[
 v_0 <  v_1 <   v_2 < \dots < v_{n-1}   < v_n.
\]
Note that this does not depend on the local orders at the vertices.
\end{remark}

\begin{remark}
Let $T=(V,E)$ be an oriented tree with a local order at every vertex. 
Let $T' =(V', E')$ be the barycentric subdivision of $T$, in which the midpoint of each edge of $T$ is viewed as a vertex in $T'$. 
Formally, we introduce new sets $E_0$, $E_-$, and $E_+$, given with bijections $E \to E_0$, $e \mapsto e_0$, $E \to E_-$, $e \mapsto e_-$, and $E \to E_+$, $e \mapsto e_+$. 
We then take disjoint unions $V' = V \sqcup E_0$ and
$E' = E_- \sqcup E_+$.  For each $e \in E$, we set 
$\operatorname{o}(e_-)= \operatorname{o}(e)$, 
$\operatorname{t}(e_-)= e_0$, $\operatorname{t}(e_+) = e_0$, and  $\operatorname{o}(e_+)= \operatorname{t}(e)$. 
It is not difficult to check that $T'$ is an oriented tree. 
The local order at a vertex $v$ in $T$ induces a local order at the same vertex $v$ viewed in $T'$. 
For the vertex $e_0$ in $T'$,  we take the local order  given by $e_- \prec_{e_0} e_+$. 
Now Theorem~\ref{t:tree-order} applied to $T'$ gives a total order on $V'$, and this restricts to the  total order on $V$ given by Theorem~\ref{t:tree-order} applied to $T$. 
Notice that $V'$ is in bijective correspondence with  the set that underlies $T$, $V \sqcup E$. 
In summary, Theorem~\ref{t:tree-order} gives a total order on $T$.
\end{remark}

\section{Ordering free products}\label{s:order-products}

In this section,  we  incorporate group actions into the results of the previous section.
We then consider a special Bass-Serre tree to obtain an easy-to-state order for free products.
We shall describe the positive cone in Subsection~\ref{subsec:explicitfreeproduct}.

\begin{proposition}\label{p:action-order}
Let $G$ be a group, $T=(V,E)$ an oriented tree with a local  order at every vertex, and $G \acts T$ a left action by orientation preserving tree automorphisms.

{\rm(a)} If the action $G \acts T$ preserves the local orders (for all $g \in G$, $v \in V$ and edges $e$ and $f$ adjacent to $v$, if $e \preceq_v f$ then $ge \preceq_{gv} gf$), then it preserves the order induced on the set $V$, i.e., for all $ g \in G$ and $x,y \in  V$,
\[
 x \leq y \implies gx \leq gy.
\]

{\rm (b)} If the stabilizer of some vertex $x$  is trivial, then
\[
 g \leq_G h \iff gx \leq hx
\]
defines a left-invariant order $\leq_G$ on $G$.
\end{proposition}

\begin{proof}
(a) For every geodesic $p$ in $T$, the geodesic $gp$ has the same edge-rise index and vertex-rise index as the geodesic $p$, since the action preserves the edge orientations and the local orders. 
Therefore, the action preserves the rise index and, consequently, the induced order on the vertex set.

(b) Clear, since $Gx$ is a totally ordered set on which $G$ acts freely and by order-preserving transformations.
\end{proof}

\begin{remark}  If a group $G$ acts on a tree $T$ such that each edge stabilizer  is trivial and each vertex stabilizer  is left-orderable, it is not difficult to show that orientations and local orders can be defined satisfying the conditions in (a). 
It follows by Bass-Serre theory that if a group $G$ is the fundamental group of a graph of groups with trivial edge groups and left-orderable vertex groups, then $G$ is  left-orderable.
Notice that we may ensure the existence of a vertex with trivial stabilizer by adding a vertex with trivial vertex group joined to any existing vertex by an added edge with trivial edge group; this does not change the fundamental group of any graph of groups.
\end{remark}

We now expand on the foregoing remark to obtain an explicit ordering in the case of free products.

Let $I$ be an indexing set and $\GG= ((G_i,\leq_i) \mid i \in I)$ a family of left-ordered groups. Assume that $I$ does not contain 0, let $I_0 = I \cup\{0\}$ and $G_0$ be the trivial group. In order to avoid technical and
notational difficulties, it is assumed that the groups in the family $\GG$ are disjoint.

The free product $G=*_{i \in I} G_i$ can be realized as the fundamental group of the tree of groups 
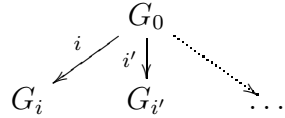
\[
\xymatrix@R=15pt{
 &  G_0 \ar[dl]_{i} \ar[d]_{i'} \ar@{..>}[dr]^{} &
 \\
 G_i & G_{i'} & \dots
}
\]
in which the vertex groups are those already indexed by $I_0$, the edge groups are trivial and indexed by $I$, and the edge indexed by $i \in I$ connects the vertex indexed by $0$ to the vertex indexed by $i$. 
The Bass-Serre tree $T=(V,E)$ corresponding to this tree of groups can be described as follows. 
The set of vertices
\[
 V = \bigsqcup_{i \in I_0} \{gG_i \mid g \in G \}
\]
is the disjoint union of the sets of left cosets of the vertex groups, the set of edges is
\[ 
 E= G \times I, 
\]
and each edge $(g,i) \in E$ connects $g=gG_0$ to $gG_i$.
For $i \in I_0$,  the vertices $gG_i$ are called $i$-vertices.

A left action $G \acts T$ of the free product $G$ on the Bass-Serre tree is given by left multiplication ($h(gG_i)=hgG_i$ and $h(g,i) = (hg,i)$). The action of $G$ on $T$ is free on the 0-vertices.

Let $\leq_I$ be a total order on $I$. We define a  local order at every vertex of the Bass-Serre tree $T$ as follows.
If $g$ is one of the 0-vertices then the set of edges adjacent to $g$ is
\[
 E_{g} = \{ (g,i) \mid i \in I \},
\]
all of these edges are oriented away from $g$, and we set
\[
 (g,i) \preceq_{g} (g,i') \iff i \leq_I i'.
\]
The local order at the $i$-vertices, for $i \in I$, is defined as follows. For $g \in G$, the set of edges adjacent to $gG_i$ is
\[
 E_{gG_i} = \{ (gh,i) \mid h \in G_i \},
\]
all of these edges are oriented toward $gG_i$, and we set
\[
 (gh,i) \preceq_{gG_i} (gh',i) \iff h \leq_i h';
\]
this definition does not depend on the choice of representative $g$ of $gG_i$, since $\leq_i$ is left $G_i$-invariant.
Thus, the local order at the 0-vertices is induced by the total order $\leq_I$ on $I$, and at the $i$-vertices,
for $i \in I$, by the total order $\leq_i$ on $G_i$.

\begin{proposition}\label{p:action-is-good}
The left action $G \acts T$ of the free product $G$ on the Bass-Serre tree preserves the edge orientations and the local order at every vertex.
\end{proposition}

\begin{proof}
We only need to prove that the action preserves the local orders.

Let $i \in I$ and $g \in G$. 
Consider two edges adjacent to the vertex $gG_i$; these are then of the form $(gh_1,i)$ and $(gh_2,i)$  with $h_1,h_2 \in G_i$.  
For any $g' \in G$,
\[
 (gh_1,i) \preceq_{gG_i} (gh_2,i) \implies h_1 \leq_i h_2 \implies (g'gh_1,i) \preceq_{g'gG_i} (g'gh_2,i).
\]
Therefore, the action of $G$ preserves the local orders at all $i$-vertices.

Let $g$ be a 0-vertex and consider the action of any $g' \in G$. 
The action moves the vertex $g$ to $g'g$ and an edge $(g,i) \in E_{g}$ to the edge $(g'g,i) \in E_{g'g}$. 
For $i_1,i_2 \in I$,
\[
 (g,i_1) \preceq_{g} (g,i_2) \implies i_1 \leq_I i_2 \implies (g'g,i_1) \preceq_{g'g} (g'g,i_2).
\]

Therefore, the action of $G$ preserves the local orders at all vertices.
\end{proof}

\begin{theorem}[Vinogradov]\label{t:free-product}
The free product $G=*_{i\in I} G_i$ of a family $(G_i \mid i \in I)$ of left-orderable groups is left-orderable.
\end{theorem}

\begin{proof}
Follows directly from Proposition~\ref{p:action-is-good}, Proposition~\ref{p:action-order}, and the fact that the stabilizer of the vertex 1 (vertex $1G_0$) in the Bass-Serre tree is trivial. 
\end{proof}

\section{Examples}

\subsection{An explicit order on the free product $G=*_{i \in I} G_i$}\label{subsec:explicitfreeproduct}

We continue the discussion of the free product $G=*_{i \in I} G_i$ of the family of left-ordered groups $\GG=((G_i,\leq_i)\mid i \in I)$ and we provide a concrete description of the left-invariant order $\leq$ on $G$
that extends the given orders on the factors and is implicit  in the proof of Theorem~\ref{t:free-product}. 
A total order $\leq_I$ on $I$ is assumed.

\begin{definition}\label{d:tau}
Define a weight function $\tau:G \to \ZZ$ as follows. 
Recall that every element $g$ of the free product $G=*_{i\in I} G_i$ can be written uniquely in the normal form
\[
 g = g_1g_2g_3 \dots g_n,
\]
where $g_j$ is a nontrivial element of $G_{i_j}$, for $j=1,\dots,n$,  and $i_j \ne i_{j+1}$, for $j = 1,\dots,n-1$.
The factors in the normal form are called the syllables of $g$, and
$n$ is called the syllable length of $g$. 
The syllable length of the trivial element is 0 and its normal form is the
empty word (denoted by 1). 
Let
\begin{align*}
 \tau(g) =
  &\#(\textup{positive syllables in }g)
  - \#(\textup{negative syllables in }g) + \\
  &\#(\textup{index jumps in }g) - \#(\textup{index drops in }g),
\end{align*}
where an index jump occurs at $j$ in $g$, for $j=1,\dots,n-1$, if  $i_j <_I i_{j+1}$ and an index drop occurs at $j$ in $g$ if $i_j >_I i_{j+1}$.
\end{definition}

\begin{proposition}
Let $G=*_{i\in I} G_i$ be the free product of a family $\GG=((G_i,\leq_i) \mid i \in I)$ of left-ordered groups.

The relation $\leq$ defined on $G$ by
\[
 g \leq h \quad \iff \quad \tau(g^{-1}h) \geq 0,
\]
is a left-invariant order on $G$, which extends the given orders on the factors.
\end{proposition}

\begin{proof}
It is sufficient to show that $\tau(g) = \rise(1,g)$, where $\rise: G \times G \to \ZZ$ is the rise index function for
the Bass-Serre tree $T=(V,E)$ of $G$, as defined in Section~\ref{s:order-products}, restricted to the set of 0-vertices.

Let the  normal form of $g$ be $g=g_{1}g_{2} \dots g_{n}$, where $g_j$ is a nontrivial element of $G_{i_j}$, for $j = 1,\ldots,n$. 
The geodesic $p_{1,g}$ from the vertex 1 to the vertex $g$ in the Bass-Serre tree has length $2n$ and has the form
\begin{multline*}
  {1} \xrightarrow{(1,i_1)}
  {G_{i_1}} \xleftarrow{(g_1,i_1)}
  {g_1} \xrightarrow{(g_1,i_2)}
  {g_1G_{i_2}} \xleftarrow{(g_1g_2,i_2)}
  {g_1g_2} \xrightarrow{} {\dots} \\
  {\dots} \xleftarrow{}
  {g_1\dots g_{n-1}} \xrightarrow{(g_1\dots g_{n-1},i_n)}
  {g_1\dots g_{n-1}G_{i_n}} \xleftarrow{(g_1\dots g_n,i_n)}
  {g}
\end{multline*}
The edge-rise index of this geodesic is 0. 
Thus $\rise(1,g) = \risev(p_{1,g})$. The vertex-rise contributions of the
0-vertices along the geodesic are positive when the index goes up (with respect to $\leq _I$), and negative when the index goes
down. 
The vertex-rise contribution at $G_{i_1}$ is positive if and only if $g_1$ is a positive element of $G_{i_1}$.
Similarly, the vertex-rise contribution at $g_1G_{i_2}$ is positive if and only if $g_2$ is a positive element in
$G_{i_2}$, and so on.
Therefore $\rise(1,g) = \risev(p_{1,g})=\tau(g)$ and this shows that $\leq$ is a total order on $G$.

If $g,g' \in G_i$, for some $i \in I$, and $g <_i g'$, then $g^{-1}g'$ is a positive syllable in $G_i$, $\tau(g^{-1}g') = 1$  and $g < g'$. 
Therefore $\leq$ extends the order relation on $G_i$.
\end{proof}

\subsection{Orders on free groups} \label{subsec:explicitfreegroup}

Let $F_k$ be the free group of rank $k$, $k \geq 2$, with base $\AA_k=\{a_1,a_2,\dots,a_k\}$. We indicate how the orders on $F_k$ defined in~\cite{sunic:free-lex} (one for each $k$) and in~\cite{sunic:from-oriented} ($(2k)!$ orders for each $k$) can be obtained from our construction.

Consider the alphabet $\AA_k^\pm = \{a_1,\dots,a_k,a_1^{-1},\dots,a_k^{-1}\}$. 
The $(2k)!$ orders on $F_k$
given in~\cite{sunic:from-oriented} are parameterized by the $(2k)!$ words over the alphabet $\AA_k^\pm$ that use each letter exactly once (in particular, the order on $F_k$ from~\cite{sunic:free-lex} corresponds to the defining word
$u=a_1 \dots a_k a_k^{-1} \dots a_1^{-1}$). 
Such a word $u$ induces a total order $\preceq_u$ on $\AA_k^\pm$ by setting
$x \prec_u y$ if and only if $x$ appears to the left of $y$ in $u$; thus,  $\preceq_u$ is just an arbitrary total order on $\AA_k^\pm$. 
For a reduced group word $w$, let $\#_{w}(g)$ denote the number of occurrences of $w$ in the reduced expression for $g$.
Given a defining word $u$, define a weight function $\tau_u: F_k \to \ZZ$ by setting
\[
 \tau_u(g) = \tau_u'(g) + \omega(g),
\]
where, for a reduced group word $g \in F_k$,
\begin{equation*}
 \tau_u'(g) = 2\left(
 \sum_{\substack{a,b \in \AA_k \\ a^{-1} \prec_u b^{-1}}} \#_{ab^{-1}}(g) -
 \sum_{\substack{a,b \in \AA_k \\ b \prec_u a}} \#_{a^{-1}b}(g) +
 \sum_{\substack{a,b \in \AA_k \\ a^{-1} \prec_u b}} \#_{ab}(g) -
 \sum_{\substack{a,b \in \AA_k \\ b^{-1} \prec_u a}} \#_{a^{-1}b^{-1}}(g)
 \right) \\
\end{equation*}
and
\begin{equation*}
 \omega(g) =
  \begin{cases}
   1, & \text{if the last letter of } g \text{ is positive (i.e., it is in } \AA_k) \\
  -1, & \text{if the last letter of } g \text{ is negative (i.e., it is in } \AA_k^{-1})\\
   0, & \text{if } g \text{ is trivial.}
  \end{cases}
\end{equation*}
The weight function $\tau_u$ defines a left-invariant order $\leq_u$ on $F_k$ with positive cone
\[
 P_u = \{ g \in F_k \mid \tau_u(g) > 0 \}.
\]

We now realize the same order by using an appropriate order on the right Cayley graph $\Gamma_k=(V,E)$ of $F_k$ with respect to $\AA_k$. 
The Cayley graph $\Gamma_k$ is an oriented tree in which every vertex $g$ (element of $F_k$) has $2k$ adjacent edges, $k$ outgoing edges with labels $a_1,\dots,a_k$ and $k$ incoming edges labeled by $a_1,\dots,a_k$
(the outgoing edge labeled by $a \in \AA_k$ connects $g$ to $ga$). 
The order $\preceq_u$ on $\AA_k^\pm$ induces an order $\preceq_g$ on the $2k$ edges adjacent to the vertex $g$ in $\Gamma_k$ by identifying the outgoing edges with labels $a_1,\dots,a_k$ with the letters $a_1,\dots,a_k$, respectively, and the incoming edges with labels
$a_1,\dots,a_k$ with the letters $a_1^{-1},\dots,a_k^{-1}$, respectively.

The left action $F_k \acts \Gamma_k$ preserves the edge orientations and the local orders at every vertex, thus it preserves the induced order $\leq$ on the tree. 
Since the action is free on the set of vertices, $F_k$ inherits the left-invariant order from the set of vertices of the tree $\Gamma_k$. 
We claim that $\rise(1,g) = \tau_u(g)$, for $g \in G$, which means that the order induced on $F_k$ from the tree $\Gamma_k$ is the same as the order $\leq_u$.

Indeed, let $g$ be a  reduced group word over $\AA_k$. 
Every pair of consecutive edges in the geodesic $p_{1,g}$ from 1 to $g$ in the tree $\Gamma_k$ comes in one of the following four types (depending on the edge orientations)
\[
 {\stackrel{\mkern10mu va^{-1}}{\bullet}} \mkern-10mu
 \xrightarrow[a]{\phantom{wrt}}
 {\stackrel{v}{\bullet}}
 \xleftarrow[b]{\phantom{wrt}}\mkern-10mu
 {\stackrel{\mkern10mu vb^{-1}}{\bullet}}
 \qquad
 {\stackrel{va}{\bullet}}
 \xleftarrow[a]{\phantom{wrt}}
 {\stackrel{v}{\bullet}}
 \xrightarrow[b]{\phantom{wrt}}
 {\stackrel{vb}{\bullet}}
 \qquad
 {\stackrel{\mkern10mu va^{-1}}{\bullet}} \mkern-10mu
 \xrightarrow[a]{\phantom{wrt}}
 {\stackrel{v}{\bullet}}
 \xrightarrow[b]{\phantom{wrt}}
 {\stackrel{vb}{\bullet}}
 \qquad
 {\stackrel{va}{\bullet}}
 \xleftarrow[a]{\phantom{wrt}}
 {\stackrel{v}{\bullet}}
 \xleftarrow[b]{\phantom{wrt}} \mkern-10mu
 {\stackrel{\mkern10mu vb^{-1}}{\bullet}}
\]
for some $a,b \in \AA_k$ and some $v \in F_k$ (we are assuming that the edges drawn on the left appear earlier in the geodesic  $p_{1,g}$).

Consider a pair of edges of the first type, corresponding to an occurrence of $ab^{-1}$ in $g$. 
The first edge (the one labeled by $a$) contributes 1 to the edge-rise index. This contribution can be canceled or doubled by the vertex-rise contribution at $v$ and the doubling occurs if and only if there is a rise at $v$, which is equivalent to the condition that $a^{-1} \prec_u b^{-1}$.
Thus, the pairs of edges of the first type (type $ab^{-1}$) correspond
to the first summation term in $\tau'_u(g)$.

Similarly, the pairs of edges of the second, third, and fourth type, regarding occurrences of $a^{-1}b$, $ab$ and $a^{-1}b^{-1}$ in $g$,  correspond to the second, third, and fourth summation term in $\tau'_u(g)$.

Therefore, $\tau'_u(g)$ is equal to the sum of the edge-rise contributions of all but the last edge and all vertex-rise contributions along $p_{1,g}$. Since $\omega(g)$ is equal to the edge-rise contribution of the last edge in the geodesic $p_{1,g}$, we obtain that $\rise(1,g)=\tau'_u(g)+\omega(g)=\tau_u(g)$.

\subsection{Some well-known orders on rooted trees} \label{subsec:rooted}

There are several well-known orders on rooted trees that are used to traverse all the vertices of a rooted tree in an organized fashion. 
Two common versions of the depth-first search on a rooted tree, often used in computer science, are the top-left-right (also known as pre-order) search and the left-right-top (also known as post-order) search.

Let us describe the total order on vertices of a rooted tree associated to the top-left-right depth-first search. 
It is assumed that the children of every vertex are totally ordered. 
It is common to draw/imagine a rooted tree with total orders on the children of every vertex embedded in the plane, with the root on top, its children on a line below it, drawn from left to right in increasing order, then the children of the children on a yet lower line, drawn from
left to right in order under their parents, respectively, and so on. 
The top-left-right order of traversing the tree in Figure~\ref{f:tlr} is indicated by the numerical labels at the vertices.

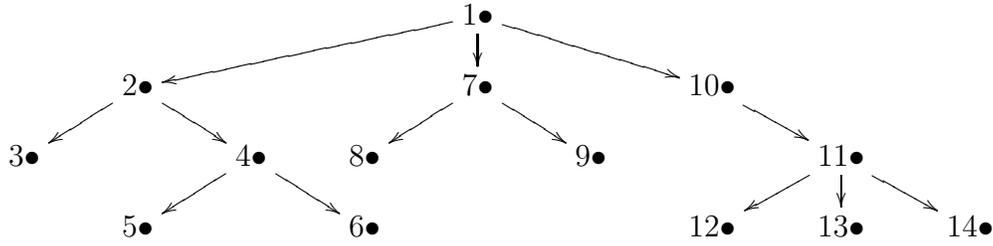
\begin{figure}[!ht]
\[
\xymatrix@R=12pt{
 &&&& 1\bullet \ar[dlll] \ar[d] \ar[drr] &&&&
 \\
 &
 2\bullet \ar[dl] \ar[dr] &&&
 7\bullet \ar[dl] \ar[dr] &&
 10\bullet \ar[dr] &&
 \\
 3\bullet &&
 4\bullet \ar[dl] \ar[dr] &
 8\bullet &&
 9\bullet &&
 11\bullet \ar[dl] \ar[d] \ar[dr] &
 \\
 &5\bullet&&6\bullet&&&12\bullet&13\bullet&14\bullet
}
\]
\caption{Top-left-right order}
\label{f:tlr}
\end{figure}

The top-left-right order may be described as follows. 
Each vertex comes in the order before all of its descendants (this is why this order is also called pre-order). 
If $v_1$ is smaller than (to the left of) its sibling $v_2$, then $v_1$ and all of its descendants are smaller than $v_2$ and all of its descendants.

In the left-right-top order each vertex comes in the order after all of its descendants (this is why this order is also called post-order), while the relative order between the descendants of the children of any vertex is the same as in the top-left-right order. 
The left-right-top order of traversing the tree in Figure~\ref{f:lrt} is indicated by the numerical labels at the vertices.

\begin{figure}[!ht]
\[
\xymatrix@R=12pt{
 &&&& \bullet14 \ar@{<-}[dlll] \ar@{<-}[d] \ar@{<-}[drr] &&&&
 \\
 &
 \bullet5 \ar@{<-}[dl] \ar@{<-}[dr] &&&
 \bullet8 \ar@{<-}[dl] \ar@{<-}[dr] &&
 \bullet13 \ar@{<-}[dr] &&
 \\
 \bullet1 &&
 \bullet4 \ar@{<-}[dl] \ar@{<-}[dr] &
 \bullet6 &&
 \bullet7 &&
 \bullet12 \ar@{<-}[dl] \ar@{<-}[d] \ar@{<-}[dr] &
 \\
 &\bullet2&&\bullet3&&&\bullet9&\bullet10&\bullet11
}
\]
\caption{Left-right-top order}
\label{f:lrt}
\end{figure}
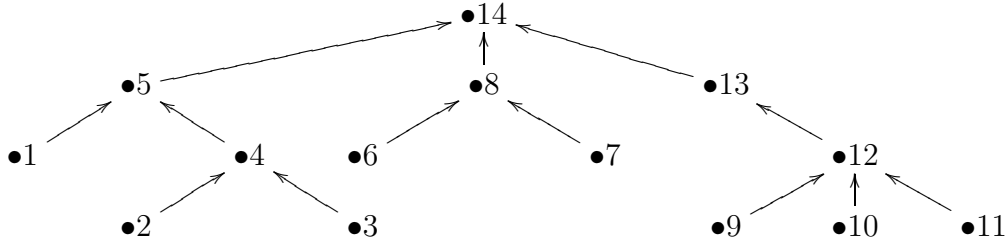

Here is a more formal description of the top-left-right, the left-right-top and a whole class of other orders on rooted
trees (including the left-top-right order, also known as in-order on binary rooted trees). 
Let $T=(V,E)$ be a rooted tree. 
For every vertex $v \in V$, let $\preceq_v$ be a total order on the set of vertices $V_v$ consisting of $v$ and all of its children (for the top-left-right order, the parent $v$ is always the smallest in $V_v$, and for the left-right-top the parent is always the largest). 
A total order $\leq$ on $V$ extending each of the orders $\preceq_v$ may be defined as follows. 
Let $x,y \in V$ be two distinct vertices and let
\[
 p_{xy} = x~e_1~v_1~e_2~v_2~\dots~v_{n-1}~e_n~y
\]
be the unique geodesic from $x$ to $y$. 
If $x$ is a descendant of $y$, then $v_{n-1}$ is a child of $y$ and we set
$x < y$ if and only if $v_{n-1} \prec_y y$. 
If none of $x$ and $y$ is a descendant of the other, then there exists three vertices $v_x$, $v$ and $v_y$ and two edges $e_x$ and $e_y$ such that
\[
 v_x~e_x~v~e_y~v_y
\]
is a piece of the geodesic $p_{xy}$ and $v$ is the closest vertex to the root on $p_{xy}$. 
Then $v_x$ and $v_y$ are two distinct children of $v$ and we set $x < y$ if and only if $v_x \prec_v v_y$.

Less formally, a child that is smaller than its parent and all of its descendants come before the parent, a child that is greater than its parent and all of its descendants come after the parent, and a smaller sibling and all of its descendants come before a greater sibling and all of its descendants.

As written, it is not even immediately obvious that $\leq$ is an order. 
However, the following proposition shows that the relation $\leq$ is equal to an order on a tree realized through a rise index. 

\begin{proposition}
Each relation $\leq$ on a rooted tree $T=(V,E)$ described above (extending given orders $\preceq_v$ for $v \in V$) can be realized as the order induced by appropriately chosen edge orientations and local orders $\preceq_v'$ on the edge sets $E_v$, for $v \in V$.
\end{proposition}

\begin{proof}
Let $v$ be a vertex in $V$ and let
\[
 \ell_1 \prec_v \ell_2 \prec_v \dots \prec_v \ell_m \prec_v v \prec_v r_1 \prec_v r_2
\prec_v \dots \prec_v r_n
\]
be the total order $\preceq_v$ on the set $V_v$. 
Orient the edges $e_{\ell_1},\dots,e_{\ell_n}$ between the parent $v$ and the children $\ell_1,\dots,\ell_m$ that are smaller than $v$ toward $v$, and the edges $e_{r_1},\dots,e_{r_m}$ between $v$ and the children $r_1,\dots,r_n$ that are greater than $v$ toward the children. 
Let $e_{v}$ be the edge from $v$ to its parent (if it exists).
Define a local order $\preceq_v'$ on $E_v$ by setting
\[
 e_{r_1} \prec_v' e_{r_2} \prec_v' \dots \prec_v' e_{r_n}  \prec_v' e_v \prec_v'  e_{\ell_1}
\prec_v' e_{\ell_2} \prec_v' \dots \prec_v' e_{\ell_m},
\]
with the understanding that $e_v$ should be omitted if $v$ is the root.

Let $\leq'$ be the order on $V$ induced by the edge orientations and the local orders $\preceq_v'$ on $E_v$, for $v \in V$.

Let $x$ and $y$ be two distinct vertices such that $x$ is a descendant of $y$ and
\[ p_{xy}=x~e_1^{\varepsilon_1}~v_1~e_2^{\varepsilon_2}~v_2~\dots~v_{n-1}~e_n^{\varepsilon_n}~y.
\]
We claim that $\rise(p_{xy})=\varepsilon_n$. 
This follows from the fact that the edge-rise contribution of $e_i$ and
the vertex-rise contribution of $v_i$ cancel each other for $i=1,\dots,n-1$. 
Indeed, if the edge-rise contribution of $e_i$ is positive then the child $v_{i-1}$ (with the understanding $v_0=x$, when $i=1$) is smaller than the
parent $v_i$ under $\preceq_{v_i}$, which means that the parent edge $e_{i+1}$ is smaller than the child edge $e_i$ under $\preceq_{v_i}'$, i.e., the vertex-rise contribution at $v_i$ is negative. 
Similarly, if the edge-rise contribution of $e_i$ is negative, the vertex-rise contribution at $v_i$ is positive. 
Since $\rise(p_{xy})=\varepsilon_n$, $x <' y$ if and only if $\varepsilon_n$ is 1, which is equivalent to $v_{n-1} \prec_y y$.

Let $x$ and $y$ be two distinct vertices none of which is a descendant of the other and let
\[
 v_x~e_x^{\varepsilon_x}~v~e_y^{\varepsilon{y}}~v_y
\]
be the piece of the geodesic $p_{xy}$ such that $v$ is the closest vertex to the root on $p_{xy}$. 
We have
\[
 \rise(x,y) = \rise(x,v) + \rise(v,y) + \rho_v = \rise(x,v) - \rise(y,v) + \rho_v = \varepsilon_x -
(- \varepsilon_y) + \rho_v = \varepsilon_x +\varepsilon_y + \rho_v,
\]
where $\rho_v$ is the vertex-rise contribution at $v$ along
$p_{xy}$. 
If $\varepsilon_x$ and $\varepsilon_y$ cancel, then both children $v_x$ and $v_y$ are smaller or both are greater than the parent $v$ under $\preceq_v$, and $x <' y$ if and only if $\rho_v=1$, which is equivalent to $e_x \prec_v' e_y$, and this to $v_x \prec_v v_y$.
Otherwise, one of the children $v_x$ and $v_y$ is smaller and the other larger than $v$ under $\preceq_v$. In this case, $x<'y$ if and only if $\varepsilon_x=\varepsilon_y=1$, which happens if and only if $v_x$ is the child smaller than $v$ and $v_y$ the child larger than $v$ under $\preceq_v$, i.e., $v_x \prec_v v_y$.

Therefore, $\leq'$ and $\leq$ are the same.
\end{proof}

\begin{remark}
Note that the order associated to the standard breadth-first search on a rooted tree can also be induced by appropriate edge orientations and local orders at all vertices.

The breadth-first order can be defined as follows. 
All vertices closer to the root come before all vertices that are further from the root.
The children of any parent $v$ respect the preassigned order $\preceq_v$ on $V_v$, for $v \in V$.  
Vertices $x$ and $y$ at the same distance from the root that do not have a common parent inherit the order, recursively, from the parents of $x$ and $y$.

If, for $v \in V$,
\[
  v \prec_v r_1 \prec_v r_2 \prec_v \dots \prec_v r_n
\]
is the total order $\preceq_v$ on the set $V_v$, define a local order $\preceq_v'$ on $E_v$, for $v \in V$, by setting
\[
 e_v \prec_v' e_{r_1} \prec_v' e_{r_2} \prec_v' \dots \prec_v' e_{r_n},
\]
and orient all edges away from the root. 
The induced order is precisely the breadth-first order. 
This is an easy corollary of the observation that, if $y$ is a descendant of $x$ at distance $k$, $k \geq 1$, in the tree, then $\rise(x,y) = 2k-1$.
\end{remark}

\def\cprime{$'$}


\end{document}